\documentclass[a4paper,12pt]{article}
\usepackage{amsmath,amsthm,amssymb}
\usepackage{mathtools}
\usepackage{graphicx}
\usepackage{comment}
\usepackage{hyperref}
\usepackage[margin=1in]{geometry}

\newtheorem{corollary}{Corollary}

\newtheorem{lemma}{Lemma}
\newtheorem{observation}{Observation}

\newtheorem{proposition}{Proposition}
\newtheorem{remark}{Remark}
\newtheorem{theorem}{Theorem}

\DeclareMathOperator*{\argmax}{arg\,max}

\usepackage[OT1,euler-digits]{eulervm}

\begin{document}
\title{Growth of recurrences with mixed multifold convolutions}

\author{Vuong Bui\thanks{LIRMM, Universit\'e de Montpellier, CNRS, 161 Rue Ada, 34095 Montpellier, France and UET, Vietnam National University, Hanoi, 144 Xuan Thuy Street, Hanoi 100000, Vietnam  (\href{mailto://bui.vuong@yandex.ru}{\texttt{bui.vuong@yandex.ru}})}}
\date{}
\maketitle
\begin{abstract}
Generalizing some popular sequences like Catalan's number, Schr\"oder's number, etc, we consider the sequence $s_n$ with $s_0=1$ and for $n\ge 1$,
\begin{multline*}
    s_n=\sum_{x_1+\dots+x_{\ell_1}=n-1} \kappa_1 s_{x_1}\dots s_{x_{\ell_1}} + \dots +\sum_{x_1+\dots+x_{\ell_{t'}}=n-1} \kappa_{t'} s_{x_1}\dots s_{x_{\ell_{t'}}}+\\ \max_{x_1+\dots+x_{\ell_{t'+1}}=n-1} \kappa_{t'+1} s_{x_1}\dots s_{x_{\ell_{t'+1}}} + \dots + \max_{x_1+\dots+x_{\ell_t}=n-1} \kappa_t s_{x_1}\dots s_{x_{\ell_t}},
\end{multline*}
where $x_i$ are nonnegative integers, $\ell_1,\dots,\ell_t$ are positive integers, and $\kappa_1,\dots,\kappa_t$ are positive reals. 
We show that it is possible to compute the growth rate $\lambda$ of $s_n$ to any precision.
In particular, for every $n\ge 2$,
    \[
        \sqrt[n]{\frac{\kappa^*}{\mathcal L(n-1) s_1} s_n} \le \lambda \le \sqrt[n]{3^{18\log 3 + 2\log\frac{s_1\mathcal L^2}{\kappa^*}}  n^{3\log n + 12\log 3 + \log\frac{s_1\mathcal L^2}{\kappa^*}} s_n},
    \]
    where $\mathcal L=\max_i \ell_i$ and $\kappa^*=\kappa_i$ for some $i$ with $\ell_i\ge 2$, and the logarithm has the base $\frac{\mathcal L+1}{\mathcal L}$. The constants in the inequalities are not very well optimized and serve mostly as a proof of concept with the ratio of the upper bound and the lower bound converging to $1$ as $n$ goes to infinity.
\end{abstract}

\section{Introduction}
Some of the most popular and well-studied recurrences with convolutions include Catalan's number: $C_0=1$ and for $n\ge 1$,
\[
    C_{n}=\sum _{i=0}^{n-1}C_{i}C_{n-1-i},
\]
and
Schr\"oder's number: $S_0=1$ and for $n\ge 1$,
\[
    S_{n}=S_{n-1}+\sum _{i=0}^{n-1}S_{i}S_{n-1-i}.
\]
If there are an arbitrary number $k$ of folds in the recurrence as in: $T_0=1$ and for $n\ge 1$,
\[
    T_n=\sum_{\substack{n_1\ge 0, \dots, n_k\ge 0\\ n_1+\dots+n_k=n-1}} T_{n_1}\dots T_{n_k},
\]
then studying the sequence is already nontrivial, as suggested by Graham, Knuth, Patashnik in Section $7.5$ of the popular book ``Concrete Mathematics'' \cite{graham1989concrete}. They argued that none of the standard techniques can be applied to the equations of generating functions with high degrees. Instead, to study $T_n$ they employ the generalized Raney lemma \cite{raney1960functional} and give a closed-form solution to $T_n$ with a combinatorial interpretation. Although it is a very beautiful approach, we believe that it is much harder to apply it to more general sequences.
For example, consider the following sequence $\{s_n\}_n$ with $s_0=1$ and for $n\ge 1$,
\begin{multline*}
    s_n=\sum_{a+b=n-1} 2s_a s_b + \sum_{a+b+c=n-1} 3s_as_bs_c + \sum_{a+b+c+d=n-1} 4 s_as_bs_cs_d +\\
    \max_{a+b+c+d+e=n-1} 5 s_as_bs_cs_ds_e + \max_{a+b+c+d+e+f=n-1} 6 s_as_bs_cs_ds_es_f,
\end{multline*}
where the numbers $2,3,4,5,6$ can be arbitrary in general. We do not write explicitly that $a,b,c,d,e,f$ are nonnegative integers and we apply the same practice throughout the article.

Generating function techniques will not succeed in solving this recurrence since the degree of the equation is $6$ (if we replace maximum by summation), beside the fact that the maximum operator complicates the already unpleasant situation. Therefore, it is hard (if not impossible) to obtain a closed-form growth rate of $\sqrt[n]{s_n}$ in general where the coefficients are arbitrary. However, we can approximate the growth rate to any precision. Before doing that, we generalize the formula with $s_0=1$ and for $n\ge 1$,
\begin{multline*}
    s_n=\sum_{x_1+\dots+x_{\ell_1}=n-1} \kappa_1 s_{x_1}\dots s_{x_{\ell_1}} + \dots +\sum_{x_1+\dots+x_{\ell_{t'}}=n-1} \kappa_{t'} s_{x_1}\dots s_{x_{\ell_{t'}}}+\\ \max_{x_1+\dots+x_{\ell_{t'+1}}=n-1} \kappa_{t'+1} s_{x_1}\dots s_{x_{\ell_{t'+1}}} + \dots + \max_{x_1+\dots+x_{\ell_t}=n-1} \kappa_t s_{x_1}\dots s_{x_{\ell_t}},
\end{multline*}
where $\ell_1,\dots,\ell_t$ are positive integers, and $\kappa_1,\dots,\kappa_t$ are positive reals. We assume some $\ell_i\ge 2$, otherwise it is just a linear recurrence.  We allow mixing some maximum operators since the techniques in the article still work well with them. Meanwhile, it makes the traditional approaches probably not applicable here.

Note that this generalization does not cover some popular sequences, for example, Motzkin's numbers with $M_0=1$ and for $n\ge 1$,
\[
    M_{n}=M_{n-1}+\sum _{i=0}^{n-2}M_{i}M_{n-2-i}.
\]
In other words, this article will not treat the situation where the sum of variables is $n-\delta$ for some $\delta\ne 1$ since the growth rate of $s_n$ may not be a limit, i.e., $\sqrt[n]{s_n}$ may no longer converge. It is possible that we need much more care to deal with $\limsup_{n\to\infty} \sqrt[n]{s_n}$ in that general case, for which we would avoid for the sake of simplicity.

In this article, we are interested in the exponential growth of $s_n$. It is quite straightforward to see that the limit
\[
    \lambda\coloneqq \lim_{n\to\infty} \sqrt[n]{s_n}
\]
exists. Indeed, we first have the following proposition.
\begin{proposition} \label{prop:supermulti}
    Let $i$ be so that $\ell_i\ge 2$. Denote $\kappa^*=\kappa$. For every $u\ge 0,v\ge 0$, we have
    \[
        s_u s_v \le \frac{1}{\kappa^*} s_{u+v+1}.
    \]
\end{proposition}
\begin{proof}
    Regardless of the summation/maximum operator, we have
    \[
        s_n\ge\max_{x_1+\dots+x_{\ell_i}=n-1} \kappa^* s_{x_1}\dots s_{x_{\ell_i}}.
    \] 
    It follows from setting $x_3=\dots=x_{\ell_i}=0$ that
    \[
        s_n\ge \max_{x_1+x_2=n-1} \kappa^* s_{x_1} s_{x_2}.
    \]
    In other words, for every $u,v$,
    \[
        s_{u+v+1} \ge \kappa^* s_u s_v.\qedhere
    \]
\end{proof}
Rewriting Proposition \ref{prop:supermulti}, we have for every $u,v\ge 1$,
\[
    \kappa^* s_{u+v-1} \ge \kappa^* s_{u-1} s_{v-1}.
\]
In other words, the sequence $\{s'_n\}_{n\ge 1}$ where $s'_n=\kappa^* s_{n-1}$ is supermultiplicative. The following corollary follows from Fekete's lemma \cite{fekete1923verteilung} for the sequence $s'_n$ (note that the sequence $s'_n$ is also positive).
\begin{corollary} \label{cor:lim-exists}
The following limit exists and can be expressed as:
\[
    \lim_{n\to\infty} \sqrt[n]{s_n}=\lim_{n\to\infty} \sqrt[n]{s'_n} = \sup_{n\ge 1} \sqrt[n]{\kappa^*s_{n-1}}.
\]
\end{corollary}
Although proving the lower bound $\sqrt[n]{\kappa^*s_{n-1}}$ is so simple and it can converge to $\lambda$ fairly reasonably fast, providing a rigorous upper bound is not so straightforward.
In the following main merit of the article, we show that the limit can be computed to any precision by a pair of matching lower and upper bounds.
\begin{theorem} \label{thm:main}
    Let $\mathcal L=\max_i \ell_i$ and $\kappa^*=\kappa_i$ for some $i$ with $\ell_i\ge 2$. For every $n\ge 2$, we have
    \[
        \sqrt[n]{\frac{\kappa^*}{\mathcal L(n-1) s_1} s_n} \le \lambda \le \sqrt[n]{3^{18\log 3 + 2\log\frac{s_1\mathcal L^2}{\kappa^*}}  n^{3\log n + 12\log 3 + \log\frac{s_1\mathcal L^2}{\kappa^*}} s_n},
    \]
    where the logarithm has the base $\frac{\mathcal L+1}{\mathcal L}$.
\end{theorem}

We can compute $\lambda$ to any precision since the ratio between the upper bound and the lower bound converges to $1$ as $n$ goes to infinity. Loosely speaking, the order of the ratio is roughly the $n$-th root of $n$ to $\log n$. Meanwhile, $s_n$ can be computed in polynomial time. Although the constants are not really optimized, they are merely for a proof of concept and there is still room for further improvement.

The lower bound is slightly different from the one in Corollary \ref{cor:lim-exists}, as we desire $s_n$ to appear in both lower and upper bounds. Meanwhile, the upper bound employs
more delicate techniques inspired by those in \cite{bui2024asymptotic} and \cite{bui2024growth}. In fact, the author was also inspired by \cite{barequet2019improved}. We give the proof in Section \ref{sec:proof}.

\section{Proof of Theorem \ref{thm:main}}
\label{sec:proof}
We slightly disturb the notion of ``rooted tree'' by specifying the number of branches for each vertex separately. A branch might be empty and the order of the branches matters. Even a leaf also has a specified number of branches (and all these branches are empty).

A \emph{composition tree} is a rooted tree where each vertex is assigned some positive real among $\kappa_1,\dots,\kappa_t$ so that if a vertex is assigned some $\kappa_i$, then the number of branches of the vertex is $\ell_i$. The value of the composition tree is the product of the values of all vertices.

The relation between composition trees and $s_n$ is given in the following observation.

\begin{observation} \label{obs:tree-interpretion}
    The value of $s_n$ is at most the sum of the values of all composition trees with $n$ vertices. The equality is attained if there is no maximum operator in the recurrence of $s_n$.
\end{observation}

\begin{lemma}
    Every composition tree with $n$ vertices has a subtree with the number of vertices in the interval
    \[
        \left[\frac{n-1}{\mathcal L+1}, \frac{\mathcal Ln+1}{\mathcal L+1}\right].
    \]
\end{lemma}
\begin{proof}
Consider a minimal subtree $T$ with the number of vertices at least $\frac{n-1}{\mathcal L+1}$, in the sense that every of its proper subtrees has the number of vertices less than $\frac{n-1}{\mathcal L+1}$. Since the root of $T$ has at most $\mathcal L$ subtrees, the total number of vertices 
of $T$ is less than
    \[
        \mathcal L \frac{n-1}{\mathcal L+1} + 1 = \frac{\mathcal L n-\mathcal L + \mathcal L + 1}{\mathcal L+1} = \frac{\mathcal Ln+1}{\mathcal L+1}.
    \]
The subtree $T$ satisfies the requirement.
\end{proof}
\begin{remark}
    When $n\ge 2$, every such subtree is a proper subtree.
\end{remark}

\begin{proposition} \label{prop:decompose}
    If every composition tree of $n$ vertices has a proper subtree with the number of vertices in the range $R$ then
    \[
        s_n\le \sum_{i\in R} \mathcal L (n-i) s_{n-i} s_i.
    \]
\end{proposition}
\begin{proof}
    If the proper subtree has $i$ vertices, then removing the subtree from the composition tree results in a tree with $n-i$ vertices. These $n-i$ vertices cannot have more than $\mathcal L(n-i)$ branches totally. Therefore, given a tree $T'$ of $n-i$ vertices, there should be at most $\mathcal L(n-i)$ ways to put a tree of $i$ vertices to be a branch of a vertex of $T'$ to form a composition tree of $n$ vertices. The conclusion follows by Observation \ref{obs:tree-interpretion}.
\end{proof}

We have the following direct corollaries.
\begin{corollary} \label{cor:history}
    For every $n\ge 2$, we set $R=\{1\}$ and obtain
    \[
        s_n\le \mathcal L (n-1) s_1 s_{n-1},
    \]
    where $s_1$ can be simply computed by $s_1=\sum_i \kappa_i$. 
\end{corollary}

\begin{corollary} \label{cor:rephrase-lower-bound}
\[
    \lim_{n\to\infty} \sqrt[n]{s_n}=\sup_{n\ge 2} \sqrt[n]{\frac{\kappa^*}{\mathcal L(n-1) s_1} s_n}.
\]
\end{corollary}
\begin{proof}
    By Corollary \ref{cor:lim-exists},
    \[
        \lambda =\sup_{n\ge 1}\sqrt[n]{s'_n} =\sup_{n\ge 2}\sqrt[n]{s'_n} = \sup_{n\ge 2} \sqrt[n]{\kappa^* s_{n-1}} \\
        \ge \sup_{n\ge 2} \sqrt[n]{\frac{\kappa^*}{\mathcal L(n-1) s_1} s_n},
    \]
    where the inequality is due to Corollary \ref{cor:history}. Note that $\sup_{n\ge 1}\sqrt[n]{s'_n} =\sup_{n\ge 2}\sqrt[n]{s'_n}$ since $\sqrt{s'_2}\ge s'_1$ by the supermultiplicativity of $s'_n$.
\end{proof}

\begin{corollary} \label{cor:decompose}
    For $n\ge 2$, we have
    \[
        s_n\le \mathcal L n^2 s_j s_{n-j}
    \]
    where
    \[
        j = \argmax\limits_{i \in \left[\frac{n-1}{\mathcal L+1}, \frac{\mathcal Ln+1}{\mathcal L+1}\right]} s_i s_{n-i}.
    \]
\end{corollary}

We can now prove an inequality that suggests some certain transform of $s_n$ can be submultiplicative, in contrast to the supermultiplicative form in Proposition \ref{prop:supermulti}.

\begin{proposition}
\label{prop:submulti}
    For every $n\ge 1$ and any $m$ with $0\le m\le n$, we have
    \[
        s_n\le \left(\frac{s_1\mathcal L^2n^3}{\kappa^*}\right)^{\alpha\log n} s_{m} s_{n-m}
    \]
    where $\alpha=\left(\log\frac{\mathcal L+1}{\mathcal L}\right)^{-1}$.
\end{proposition}
Note that the exponent can be written as $\log n/\log \frac{\mathcal L+1}{\mathcal L}$. In other words, it is the logarithm of $n$ with base $\frac{\mathcal L+1}{\mathcal L}$. Therefore, the notation $\log$ here can have any base.
\begin{proof}
    The conclusion holds for $n\le 2$ (note that $s_2\le \mathcal L(s_1)^2$). Following induction method, we prove that the conclusion holds for every $n\ge 3$ given that it also holds for any smaller number than $n$. Note that we only need to prove the conclusion for $m\le \frac{n}{2}$ since otherwise, we simply set $m=n-m$.

    By Corollary \ref{cor:decompose}, we have
    \[
        s_n\le \mathcal L n^2 s_j s_{n-j}
    \]
    for some $j\in \left[\frac{n-1}{\mathcal L+1}, \frac{\mathcal Ln+1}{\mathcal L+1}\right]$.

    Without loss of generality, we suppose that $j\le n-j$, otherwise we set $j=n-j$. It follows that $m\le \frac{n}{2}\le n-j$.
    \begin{itemize}
        \item If $j < m\le n-j$, then applying Corollary \ref{cor:history} with $n-j\ge 2$, we have
        \begin{align*}
            s_{n-j}&\le \mathcal L (n-j-1) s_1 s_{n-j-1} \\
            &\le \mathcal Ln s_1 s_{n-j-1} \\
            &\le \mathcal Ln s_1  \left(\frac{s_1\mathcal L^2(n-j-1)^3}{\kappa^*}\right)^{\alpha\log(n-j-1)} s_{n-m} s_{m-j-1} &\text{by induction hypothesis}\\
            &\le \mathcal Ln s_1  \left(\frac{s_1\mathcal L^2n^3}{\kappa^*}\right)^{\alpha\log(n-j-1)} s_{n-m} s_{m-j-1}\\
            &\le \mathcal Ln s_1  \left(\frac{s_1\mathcal L^2n^3}{\kappa^*}\right)^{\alpha\log\frac{\mathcal Ln}{\mathcal L+1}} s_{n-m} s_{m-j-1}.
        \end{align*}
        Note that $n-j\ge\frac{n}{2}\ge \frac{3}{2}\implies n-j\ge 2\implies n-j-1\ge 1$, meanwhile $n-m\ge 0$ and $m-j-1\ge 0$, hence we can apply induction hypothesis.
        
        Totally,
        \begin{align*}
            s_n&\le \mathcal L n^2 s_j s_{n-j}\\
            &\le \mathcal Ln^2 \mathcal Ln s_1  \left(\frac{s_1\mathcal L^2n^3}{\kappa^*}\right)^{\alpha\log\frac{\mathcal Ln}{\mathcal L+1}} s_{n-m} s_{m-j-1} s_j\\
            &\le s_1\mathcal L^2n^3 \left(\frac{s_1\mathcal L^2n^3}{\kappa^*}\right)^{\alpha\log\frac{\mathcal Ln}{\mathcal L+1}} s_{n-m} \frac{1}{\kappa^*} s_m &\text{by Proposition \ref{prop:supermulti}}\\
            &=\left(\frac{s_1\mathcal L^2n^3}{\kappa^*}\right)^{\alpha(\frac{1}{\alpha} + \log n - \frac{1}{\alpha})} s_{n-m} s_m\\
            &= \left(\frac{s_1\mathcal L^2n^3}{\kappa^*}\right)^{\alpha\log n} s_m s_{n-m}.
        \end{align*}
        \item If $m\le j$, then by induction hypthesis,
        \begin{align*}
            s_j&\le \left(\frac{s_1\mathcal L^2j^3}{\kappa^*}\right)^{\alpha\log j} s_m s_{j-m} \\
            &\le \left(\frac{s_1\mathcal L^2n^3}{\kappa^*}\right)^{\alpha\log \frac{\mathcal Ln}{\mathcal L+1}} s_m s_{j-m}.
        \end{align*}
        Note that $j\le\frac{n}{2}\implies j\le \frac{\mathcal Ln}{\mathcal L+1}$.

        Totally,
        \begin{align*}
            s_n&\le \mathcal L n^2 s_j s_{n-j}\\
            &\le \mathcal Ln^2 \left(\frac{s_1\mathcal L^2n^3}{\kappa^*}\right)^{\alpha\log \frac{\mathcal Ln}{\mathcal L+1}} s_m s_{j-m} s_{n-j}\\
            &\le \mathcal Ln^2 \left(\frac{s_1\mathcal L^2n^3}{\kappa^*}\right)^{\alpha\log\frac{\mathcal Ln}{\mathcal L+1}} s_m \frac{1}{\kappa^*} s_{n-m+1} &\text{by Proposition \ref{prop:supermulti}}\\
            &\le \mathcal Ln^2 \left(\frac{s_1\mathcal L^2n^3}{\kappa^*}\right)^{\alpha\log\frac{\mathcal Ln}{\mathcal L+1}} s_m \frac{1}{\kappa^*} \mathcal L (n-m) s_1 s_{n-m} &\text{by Corollary \ref{cor:history} with $n-m+1\ge 2$}\\
            &\le \mathcal Ln^2 \left(\frac{s_1\mathcal L^2n^3}{\kappa^*}\right)^{\alpha\log\frac{\mathcal Ln}{\mathcal L+1}} s_m \frac{1}{\kappa^*} \mathcal L n s_1 s_{n-m}\\
            &= \frac{1}{\kappa^*} s_1\mathcal L^2n^3 \left(\frac{s_1\mathcal L^2n^3}{\kappa^*}\right)^{\alpha\log\frac{\mathcal Ln}{\mathcal L+1}} s_m s_{n-m}\\
            &=\left(\frac{s_1\mathcal L^2n^3}{\kappa^*}\right)^{\alpha(\frac{1}{\alpha} + \log n - \frac{1}{\alpha})} s_m  s_{n-m} \\
            &= \left(\frac{s_1\mathcal L^2n^3}{\kappa^*}\right)^{\alpha\log n} s_m s_{n-m}. 
        \end{align*}
    \end{itemize}
    The induction step finishes the proof.
\end{proof}

We now compose a transform of $s_n$ with some certain form of submultiplicativity. Writing Proposition \ref{prop:submulti} differently, for every $\ell,m\ge 0$, $\ell+m=n$, we have
\[
    s_n\le  n^{3\alpha \log n} n^{\beta} s_\ell s_m,
\]
where $\beta=\alpha\log \frac{s_1\mathcal L^2}{\kappa^*}$ (note that $x^{\alpha\log n} = n^{\alpha\log x}$).

From now on, we only consider $\ell,m\in [\frac{n}{3},\frac{2n}{3}]$.
It means $n\le 3\ell$ and $n\le 3m$, for which we have
\begin{align*}
     n^{3\alpha\log n} n^\beta s_n&\le  n^{3\alpha \log n} n^{\beta} s_\ell  n^{3\alpha \log n} n^\beta s_m\\
    &\le  (3\ell)^{3\alpha \log (3\ell)} (3\ell)^{\beta} s_\ell  (3m)^{3\alpha \log (3m)} (3m)^\beta s_m\\
    &=  3^{3\alpha \log (3\ell)} \ell^{3\alpha \log (3\ell)}  3^{\beta} \ell^\beta s_\ell  3^{3\alpha \log (3m)} m^{3\alpha \log (3m)}  3^{\beta} m^\beta s_m\\
    &=  (3\ell)^{3\alpha \log 3}\ell^{3\alpha \log 3} \ell^{3\alpha \log \ell}   3^{\beta} \ell^\beta s_\ell  (3m)^{3\alpha \log 3}m^{3\alpha \log 3} m^{3\alpha \log m} 3^{\beta} m^\beta  s_m\\
    &=  3^{3\alpha \log 3 + \beta}\ell^{2(3\alpha \log 3)} \ell^{3\alpha \log \ell}   \ell^\beta s_\ell  3^{3\alpha \log 3 + \beta} m^{2(3\alpha \log 3)} m^{3\alpha \log m} m^\beta s_m\\
    &= 3^{2(3\alpha \log 3 + \beta)}\ell^{2(3\alpha \log 3)} [\ell^{3\alpha \log \ell}   \ell^\beta s_\ell] m^{2(3\alpha \log 3)} [m^{3\alpha \log m} m^\beta s_m].
\end{align*}

Let $f(n)= n^{3\alpha\log n} n^\beta s_n$. We have
\[
    f(n)\le 3^{2(3\alpha \log 3 + \beta)}\ell^{6\alpha \log 3} f(\ell) m^{6\alpha \log 3} f(m).
\]
It follows that
\begin{align*}
    n^{12\alpha\log 3}f(n)&\le 3^{2(3\alpha \log 3 + \beta)}n^{6\alpha\log 3}\ell^{6\alpha \log 3} f(\ell)n^{6\alpha\log 3} m^{6\alpha \log 3} f(m)\\
    &\le 3^{2(3\alpha \log 3 + \beta)}(3\ell)^{6\alpha\log 3}\ell^{6\alpha \log 3} f(\ell) (3m)^{6\alpha\log 3} m^{6\alpha \log 3} f(m)\\
    &\le 3^{2(3\alpha \log 3 + \beta)} 3^{6\alpha\log 3}\ell^{12\alpha \log 3} f(\ell) 3^{6\alpha\log 3} m^{12\alpha \log 3} f(m)\\
    &\le 3^{18\alpha \log 3 + 2\beta} \ell^{12\alpha \log 3} f(\ell) m^{12\alpha \log 3} f(m).
\end{align*}
This implies
\[
    3^{18\alpha \log 3 + 2\beta} n^{12\alpha\log 3}f(n) \le 3^{18\alpha \log 3 + 2\beta} \ell^{12\alpha \log 3} f(\ell) 3^{18\alpha \log 3 + 2\beta}m^{12\alpha \log 3} f(m).
\]
In other words, the sequence
\[
    s'_n=3^{18\alpha \log 3 + 2\beta}  n^{3\alpha\log n + 12\alpha\log 3} n^\beta s_n
\]
is weakly submultiplicative, in the sense that $s'_n\le s'_ms'_{n-m}$ for any $m\in [\frac{n}{3},\frac{2n}{3}]$. However, we can still apply the generalized Fekete lemma \cite{de1952some} to sequences with this condition and obtain
\[
    \lambda=\lim_{n\to\infty} \sqrt[n]{s_n}=\lim_{n\to\infty} \sqrt[n]{s'_n} = \inf_n \sqrt[n]{s'_n} = \inf_n \sqrt[n]{3^{18\alpha \log 3 + 2\beta}  n^{3\alpha\log n + 12\alpha\log 3 + \beta} s_n}.
\]

Combining with Corollary \ref{cor:rephrase-lower-bound}, for every $n\ge 2$, we have
\[
    \sqrt[n]{\frac{\kappa^*}{\mathcal L(n-1) s_1} s_n} \le \lambda \le \sqrt[n]{3^{18\log 3 + 2\log\frac{s_1\mathcal L^2}{\kappa^*}}  n^{3\log n + 12\log 3 + \log\frac{s_1\mathcal L^2}{\kappa^*}} s_n},
\]
where the logarithm has the base $\frac{\mathcal L+1}{\mathcal L}$.
This concludes Theorem \ref{thm:main}.
\bibliographystyle{unsrt}
\bibliography{compfold}

\end{document}